\documentclass{lms}
\usepackage{amssymb}
\usepackage{amsmath}
\usepackage[matrix,arrow,curve,cmtip]{xy}
\usepackage{mathrsfs}
\usepackage{graphicx}

\numberwithin{equation}{section}
\newtheorem{lettertheorem}{Theorem}

\newtheorem{theorem}[equation]{Theorem} 
\newtheorem{lemma}[equation]{Lemma}     
\newtheorem{cor}[equation]{Corollary}
\newtheorem{prop}[equation]{Proposition}
\newtheorem{addendum}[equation]{Addendum}

\newnumbered{definition}[equation]{Definition}
\newnumbered{remark}[equation]{Remark}
\newnumbered{example}[equation]{Example}

\newcommand{\TR}{\operatorname{TR}}
\newcommand{\N}{\mathbb{N}}
\newcommand{\Z}{\mathbb{Z}}
\newcommand{\C}{\mathbb{C}}
\newcommand{\xto}{\xrightarrow}
\newcommand{\length}{\operatorname{length}}

\title[$K$-theory of truncated polynomial algebras]{On the $K$-theory
  of truncated polynomial algebras\\ over the integers}

\author[Angeltveit, Gerhardt, \and Hesselholt]{Vigleik Angeltveit,
  Teena Gerhardt, \and Lars Hesselholt}

\classno{19D55 (primary), 55Q91 (secondary)}

\extraline{The first and third authors were supported in part by the
National Science Foundation}

\begin{document}

\maketitle

\begin{abstract}We show that $K_{2i}(\Z[x]/(x^m),(x))$ is finite of
order $(mi)!(i!)^{m-2}$ and that $K_{2i+1}(\Z[x]/(x^m),(x))$ is free 
abelian of rank $m-1$. This is accomplished by showing that the
equivariant homotopy groups $\TR_{q-\lambda}^n(\Z;p)$ of the
topological Hochschild $\mathbb{T}$-spectrum $T(\Z)$ are free abelian
for $q$ even, and finite for $q$ odd, and by determining their 
ranks and orders, respectively.
\end{abstract}

\section*{Introduction}

It was proved by Soul\'{e}~\cite{soule1} and
Staffeldt~\cite{staffeldt1} that, for every non-negative integer $q$,
the abelian group $K_q(\Z[x]/(x^m),(x))$ is finitely generated and
that its rank is either $0$ or $m-1$ according as $q$ is even or
odd. In this paper, we prove the following more precise result:

\begin{lettertheorem}\label{kthm}Let $m$ be a positive integer and $i$
a non-negative integer. Then:
\begin{enumerate}
\item[(i)] The abelian group $K_{2i+1}(\Z[x]/(x^m),(x))$ is free
of rank $m-1$.
\item[(ii)] The abelian group $K_{2i}(\Z[x]/(x^m),(x))$ is finite of
order $(mi)!(i!)^{m-2}$.
\end{enumerate}
\end{lettertheorem}

In particular, the $p$-primary torsion subgroup of
$K_{2i}(\Z[x]/(x^m),(x))$ is zero, for every prime number $p > mi$. At
present, we do not know the group structure of the finite abelian
group in degree $2i$ except for small values of $i$ and $m$. We remark
that the result agrees with the calculation by Geller and
Roberts~\cite{roberts} of the group in degree $2$.

To prove Theorem~\ref{kthm}, we use the cyclotomic trace map of
B\"{o}kstedt-Hsiang-Madsen~\cite{bokstedthsiangmadsen} from
the $K$-groups in the statement to the corresponding topological
cyclic homology groups and a theorem of McCarthy~\cite{mccarthy1}
which shows that this map becomes an isomorphism after pro-finite
completion. The third author and Madsen~\cite[Proposition~4.2.3]{hm1},
in turn, gave a formula for the topological cyclic homology groups in
question in terms of the equivariant homotopy groups 
$$\TR_{q-\lambda}^r(\Z) = 
[ S^q \wedge (\mathbb{T}/C_r)_+, S^{\lambda} \wedge T(\Z) ]_{\mathbb{T}}$$
of the topological Hochschild $\mathbb{T}$-spectrum
$T(\mathbb{Z})$. Here $\mathbb{T}$ is the multiplicative group of
complex numbers of modulus $1$, $C_r \subset \mathbb{T}$ is the
finite subgroup of the indicated order, $\lambda$ is a finite
dimensional complex $\mathbb{T}$-representation, and $S^{\lambda}$ is
the one-point compactification of $\lambda$. Since the groups
$K_q(\Z[x]/(x^m),(x))$ and $\TR_{q-\lambda}^r(\Z)$ are finitely 
generated by~\cite{soule1,staffeldt1} and Lemma~\ref{rational}, respectively, these earlier results amount to a long exact sequence 
$$\xymatrix{
{ \cdots } \ar[r] &
{ \lim_R^{\phantom{R}} \TR_{q-1-\lambda_d}^{r/m}(\Z) } \ar[r]^{V_m} &
{ \lim_R^{\phantom{R}} \TR_{q-1-\lambda_d}^r(\Z) } \ar[r] &
{ K_q(\Z[x]/(x^m),(x)) } \ar[r] &
{ \cdots } \cr
}$$
where $d = d(m,r)$ is the integer part of $(r-1)/m$, and where
$\lambda_d$ is the sum
$$\lambda_d = \C(d) \oplus \C(d-1) \oplus \cdots \oplus \C(1)$$
of the one-dimensional complex $\mathbb{T}$-representations defined by
$\C(i) = \C$ with $\mathbb{T}$ acting from the left by $z \cdot 
w = z^iw$. The two limits range over the positive integers divisible
by $m$ and the positive integers, respectively, ordered under
division. The structure maps $R$ and the map $V_m$ are explained in
Section~\ref{bigtrsection} below. In particular, we show that for every
integer $q$ there exists a positive integer $r = r(m,q)$ divisible
by $m$ such that the canonical projections 
$$\lim_R \TR_{q-\lambda_d}^r(\Z) \to \TR_{q-\lambda_d}^r(\Z), \hskip10mm
\lim_R \TR_{q-\lambda_d}^{r/m}(\Z) \to \TR_{q-\lambda_d}^{r/m}(\Z)$$
are isomorphisms.

After localizing at a prime number $p$, the abelian groups
$\TR_{q-\lambda}^r(\Z)$ decompose as products of the $p$-typical
equivariant homotopy groups
$$\TR_{q-\lambda}^n(\Z;p) = \TR_{q-\lambda}^{p^{n-1}}(\Z) =
[ S^q \wedge (\mathbb{T}/C_{p^{n-1}})_+, S^{\lambda} \wedge
T(\Z)]_{\mathbb{T}}.$$
In addition, the Verschiebung map $V_m$ which appears in the
long exact sequence above may be expressed in terms of the
$p$-typical Verschiebung map $V = V_p$. The corresponding $p$-typical
equivariant homotopy groups with $\Z/p\Z$-coefficients
$$\TR_{q-\lambda}^n(\Z;p,\Z/p\Z) =
[ S^q \wedge (\mathbb{T}/C_{p^{n-1}})_+, M_p \wedge S^{\lambda} \wedge
T(\Z) ]_{\mathbb{T}}$$
were evaluated by the first and second
author~\cite{angeltveitgerhardt} and by Tsalidis~\cite{tsalidis1}.
More generally, the first and second author~\cite{angeltveitgerhardt}
evaluated the $\operatorname{RO}(\mathbb{T})$-graded equivariant
homotopy groups 
$$\TR_{\alpha}^n(\Z;p,\Z/p\Z) =
[ S^{\beta} \wedge (\mathbb{T}/C_{p^{n-1}})_+, M_p \wedge S^{\gamma} \wedge
T(\Z) ]_{\mathbb{T}},$$
where $\alpha \in \operatorname{RO}(\mathbb{T})$ is any virtual finite
dimensional orthogonal $\mathbb{T}$-representation, and where $\beta$
and $\gamma$ are chosen actual representations with $\alpha = [\beta]
- [\gamma]$. Based on these results, we prove: 

\begin{lettertheorem}\label{trthm}Let $p$ be a prime number, let $n$ be a
positive integer, and let $\lambda$ be a finite dimensional complex
$\mathbb{T}$-representation. Then:
\begin{enumerate}
\item[(i)] For $q = 2i$ even, $\TR_{q-\lambda}^n(\Z;p)$ is a free
abelian group whose rank is equal to the number of integers $0
\leqslant s < n$ such that $i = \dim_{\C}(\lambda^{C_{p^s}})$.

\item[(ii)] For $q = 2i - 1$ odd, $\TR_{q-\lambda}^n(\Z;p)$ is a finite
abelian group whose order is determined, recursively, by
$$\lvert \TR_{q-\lambda}^n(\Z;p) \rvert = \begin{cases}
\lvert \TR_{q-\lambda'}^{n-1}(\Z;p) \rvert \cdot
p^{n-1}(i-\dim_{\C}(\lambda)) & 
\text{if\, $i > \dim_{\C}(\lambda)$ } \cr
\lvert \TR_{q-\lambda'}^{n-1}(\Z;p) \rvert & 
\text{if\, $i \leqslant \dim_{\C}(\lambda)$, } \cr 
\end{cases}$$
where
$\lambda' = \rho_p^*\lambda^{C_p}$ is the  
$\mathbb{T}/C_p$-representation $\lambda^{C_p}$ viewed as a
$\mathbb{T}$-representation via the isomorphism $\rho_p \colon
\mathbb{T} \to \mathbb{T}/C_p$ given by the $p$th root.
\item[(iii)] For every integer $q$, the Verschiebung map
$$V \colon \TR_{q-\lambda}^{n-1}(\Z;p) \to
\TR_{q-\lambda}^n(\Z;p)$$
is injective, and for $q$ even the cokernel is a free abelian group.
\end{enumerate}
\end{lettertheorem}

We remark that for $\lambda = 0$ the result is that
$$\lvert \TR_{2i-1}^n(\Z;p) \rvert = p^{n(n-1)/2} \, i^n$$
while the even groups all are zero with the exception of
$\TR_0^n(\Z;p)$ which is a free abelian group of rank $n$. In the
case $n = 1$ which was proved by B\"{o}kstedt~\cite{bokstedt1}, the
groups are all cyclic. For $n > 1$, this is not the case. It remains a
very interesting problem to determine the structure of these groups.
We refer to~\cite[Theorem~18]{h7} for some partial results.

It follows from Theorem~\ref{trthm} that with $\Z/p\Z$-coefficients the
Verschiebung map
$$V \colon \TR_{q-\lambda}^{n-1}(\Z;p,\Z/p\Z) \to
\TR_{q-\lambda}^n(\Z;p,\Z/p\Z)$$
is injective for $q$ even. We do not know the value of this map for
$q$ odd. The calculation of this map, and hence, the groups
$K_q(\Z[x]/(x^m),(x);\Z/p\Z)$ claimed
in~\cite[Proposition~7.7]{tsalidis1} is incorrect. Indeed, in
loc.~cit., it is only the map induced by $V$ between the
$E^{\infty}$-terms of two spectral sequences which is evaluated.

The paper is organized as follows. In Section~\ref{bigtrsection}, we
show that the groups $\TR_{q-\lambda}^r(\Z)$ are finitely generated
and determine their ranks. In  Section~\ref{trsection}, we recall the
results of~\cite{angeltveitgerhardt} and~\cite{tsalidis1} and prove
Theorem~\ref{trthm}. In the following Section~\ref{ksection}, we
evaluate the terms in the long exact sequence above and prove
Theorem~\ref{kthm}. Finally, in Section~\ref{dualnumberssection}, we
specialize to the case of the dual numbers and determine the structure
of the finite group $K_{2i}(\Z[x]/(x^2),(x))$ of order $(2i)!$ in low
degrees.

\section{The groups \,$\TR_{q-\lambda}^r(\Z)$}\label{bigtrsection}

In this section, we recall the groups $\TR_{q-\lambda}^r(\Z)$ and the
Frobenius, Verschiebung, and restriction operators that relate
them. We refer to~\cite[Section~1]{hm} and~\cite[Section~2]{h7} for
further details.

Let $A$ be a unital associative ring. The topological
Hochschild $\mathbb{T}$-spectrum $T(A)$ is, in particular, an
orthogonal $\mathbb{T}$-spectrum in the sense
of~\cite[Definition~II.2.6]{mandellmay}. Therefore, for every finite
dimensional orthogonal $\mathbb{T}$-representation $\lambda$ and every
finite subgroup $C_r \subset \mathbb{T}$, we have the equivariant
homotopy group given by the following abelian group of maps in the
homotopy category of orthogonal $\mathbb{T}$-spectra
$$\TR_{q-\lambda}^r(A) = [ S^q \wedge (\mathbb{T}/C_r)_+,
S^{\lambda} \wedge T(A) ]_{\mathbb{T}}.$$
For every divisor $s$ of $r$ with quotient $t = r/s$, there are maps
$$\begin{aligned}
F_s & \colon \TR_{q-\lambda}^r(A) \to \TR_{q-\lambda}^t(A) \hskip7mm
\text{(Frobenius)} \cr
V_s & \colon \TR_{q-\lambda}^t(A) \to \TR_{q-\lambda}^r(A) \hskip7mm
\text{(Verschiebung)} \cr
\end{aligned}$$
induced by maps $f_s \colon (\mathbb{T}/C_t)_+ \to (\mathbb{T}/C_r)_+$
and $v_s \colon (\mathbb{T}/C_r)_+ \to (\mathbb{T}/C_t)_+$ in the
homotopy category of orthogonal $\mathbb{T}$-spectra. The map $f_s$ is
the map of suspension $\mathbb{T}$-spectra induced by the canonical
projection $\operatorname{pr} \colon \mathbb{T}/C_t \to
\mathbb{T}/C_r$ and the map $v_s$ is the corresponding transfer map
defined as follows. Let $\iota \colon \mathbb{T}/C_t \hookrightarrow
\mu$ be an embedding into a finite dimensional orthogonal
$\mathbb{T}$-representation. Then the product embedding
$(\iota,\operatorname{pr}) \colon \mathbb{T}/C_t \to \mu \times
(\mathbb{T}/C_r)$ has trivial normal bundle, and the linear structure
of $\mu$ determines a canonical trivialization. Therefore, the
Pontryagin-Thom construction gives a map of pointed
$\mathbb{T}$-spaces
$$S^{\mu} \wedge (\mathbb{T}/C_r)_+ \to 
S^{\mu} \wedge (\mathbb{T}/C_t)_+.$$
The induced map of suspension $\mathbb{T}$-spectra determines the
homotopy class of maps of orthogonal $\mathbb{T}$-spectra $v_s \colon
(\mathbb{T}/C_r)_+ \to (\mathbb{T}/C_t)_+$ and this homotopy class is
independent of the choice of embedding $\iota$ as well as the choices
made in forming the Pontryagin-Thom construction.

The orthogonal $\mathbb{T}$-spectrum $T(A)$ has the additional
structure of a cyclotomic spectrum in the sense
of~\cite[Definition~2.2.]{hm}. This implies that, in the situation
above, there is a map
$$R_s \colon \TR_{q-\lambda}^r(A) \to
\TR_{q-\lambda'}^t(A) \hskip 7mm \text{(restriction)},$$
where $\lambda' = \rho_s^*(\lambda^{C_s})$ is the
$\mathbb{T}/C_s$-representation $\lambda^{C_s}$ considered as a 
$\mathbb{T}$-representation via the isomorphism $\rho_s \colon
\mathbb{T} \to \mathbb{T}/C_s$ defined by $\rho_s(z) = z^{1/s}C_s$. 
Moreover, the map $R_s$ admits a canonical factorization which we now
explain. In general, let $G$ be a compact Lie group and $\mathscr{F}$ a
family of closed subgroups of $G$ stable under conjugation and passage to
subgroups. We recall that a universal $\mathscr{F}$-space is a
$G$-CW-complex $E\mathscr{F}$ with the property that, for every closed
subgroup $H \subset G$, the fixed point set $(E\mathscr{F})^H$ is
contractible if $H \in \mathscr{F}$ and empty if
$H \notin \mathscr{F}$. It was proved by
tom~Dieck~\cite[Satz~1]{tomdieck} that a universal $\mathscr{F}$-space
$E\mathscr{F}$ exists and that, if both $E\mathscr{F}$ and
$E'\!\mathscr{F}$ are universal $\mathscr{F}$-spaces, then there
exists a unique $G$-homotopy class of $G$-homotopy equivalences $f \colon
E\mathscr{F} \to E'\!\mathscr{F}$. Given a universal
$\mathscr{F}$-space $E\mathscr{F}$, the pointed $G$-space
$\tilde{E}\mathscr{F}$ is defined to be the mapping cone of the map 
$\pi \colon  E\mathscr{F}_+ \to S^0$ that collapses $E\mathscr{F}$ onto
the non-base point such that we have a cofibration sequence of pointed
$G$-spaces
$$E\mathscr{F}_+ \xto{\pi}
S^0 \xto{\iota}
\tilde{E}\mathscr{F} \xto{\delta}
\Sigma E\mathscr{F}_+.$$
If $N \subset G$ is a closed normal subgroup, we denote by
$\mathscr{F}[N]$ the family of closed subgroups $H \subset G$ that do
not contain $N$ as a subgroup. Now, the map $R_s$ admits a factorization
as the composition of the map
$$\TR_{q-\lambda}^r(A) =
[S^q \wedge (\mathbb{T}/C_r)_+, S^{\lambda} \wedge
T(A) ]_{\mathbb{T}} \to
[S^q \wedge (\mathbb{T}/C_r)_+, S^{\lambda} \wedge
\tilde{E}\mathscr{F}[C_s] \wedge T(A) ]_{\mathbb{T}}$$
induced by the map $\iota \colon S^0 \to \tilde{E}\mathscr{F}[C_s]$
and a canonical isomorphism
$$[S^q \wedge (\mathbb{T}/C_r)_+, S^{\lambda} \wedge
\tilde{E}\mathscr{F}[C_s] \wedge T(A) ]_{\mathbb{T}} \xto{\sim}
[S^q \wedge (\mathbb{T}/C_t)_+, S^{\lambda'} \wedge T(A)
]_{\mathbb{T}} = \TR_{q-\lambda'}^t(A)$$
induced from the cyclotomic structure of $T(A)$
and~\cite[Proposition~V.4.17]{mandellmay}.

The group isomorphism $\rho_s \colon \mathbb{T} \to \mathbb{T}/C_s$
gives rise to an equivalence of categories $\rho_s^*$ from the
category of orthogonal $\mathbb{T}/C_s$-spectra to the category of
orthogonal $\mathbb{T}$-spectra defined by
$$(\rho_s^*T)({\lambda}) = \rho_s^*(T({(\rho_s^{-1})^*\lambda})).$$
The following result is a generalization of~\cite[Theorem~2.2]{hm}. 

\begin{prop}\label{fundamentallongexactsequence}Let $A$ be a unital
associative ring, $r$ a positive integer, and $\lambda$ a finite
dimensional orthogonal $\mathbb{T}$-representation. Let $p$ be a prime
number that divides $r$, and $u$ and $v$ positive integers with $u+v =
v_p(r) + 1$. Then there is a natural long exact sequence
$$\cdots \to
\mathbb{H}_q(C_{p^u},\TR_{\,\boldsymbol{\cdot}-\lambda}^{r/p^u}(A)) \xto{N_{p^u}}
\TR_{q-\lambda}^r(A) \xto{R_{p^v}}
\TR_{q-\lambda'}^{r/p^v}(A) \xto{\partial}
\mathbb{H}_{q-1}(C_{p^u},\TR_{\,\boldsymbol{\cdot}-\lambda}^{r/p^u}(A)) \to
\cdots$$
where the left-hand term is the $q$th Borel homology group of the
group $C_{p^u}$ with coefficients in the orthogonal
$\mathbb{T}$-spectrum defined by
$$\TR_{\,\boldsymbol{\cdot}-\lambda}^{r/p^u}(A) =
\rho_{r/p^u}^*(S^{\lambda} \wedge T(A))^{C_{r/p^u}}.$$
\end{prop}

\begin{proof}The cofibration sequence of pointed $\mathbb{T}$-spaces
$$E\mathscr{F}[C_{p^v}]_+ \xto{\pi}
S^0 \xto{\iota}
\tilde{E}\mathscr{F}[C_{p^v}] \xto{\delta}
\Sigma E\mathscr{F}[C_{p^v}]_+$$
gives rise to a cofibration sequence of orthogonal
$\mathbb{T}$-spectra
$$E\mathscr{F}[C_{p^v}]_+ \wedge T(A) \xto{\pi}
T(A) \xto{\iota}
\tilde{E}\mathscr{F}[C_{p^v}] \wedge T(A) \xto{\delta}
\Sigma E\mathscr{F}[C_{p^v}]_+ \wedge T(A).$$
This, in turn, gives rise to a long exact sequence of equivariant
homotopy groups which we now identify with the long exact sequence of
the statement. By definition, we have
$$\TR_{q-\lambda}^r(A) = 
[ S^q \wedge (\mathbb{T}/C_r)_+, 
S^{\lambda} \wedge T(A) ]_{\mathbb{T}},$$
and, as recalled above, the restriction map $R_{p^v}$ factors through
the canonical isomorphism
$$[ S^q \wedge (\mathbb{T}/C_r)_+, S^{\lambda} \wedge 
\tilde{E}\mathscr{F}[C_{p^v}] \wedge T(A) ]_{\mathbb{T}}
\xto{\sim} \TR_{q-\lambda'}^{r/p^v}(A).$$
This identifies the middle and right-hand term of the long exact
sequence. To identify the left-hand term, we recall
from~\cite[Proposition~V.2.3]{mandellmay} the change-of-groups
isomorphism
$$[ S^q \wedge (C_r/C_r)_+, S^{\lambda} \wedge E\mathscr{F}[C_{p^v}]_+
\wedge T(A) ]_{C_r}
\xto{\sim}
[ S^q \wedge (\mathbb{T}/C_r)_+, S^{\lambda} \wedge E\mathscr{F}[C_{p^v}]_+
\wedge T(A) ]_{\mathbb{T}}.$$
On the left-hand side, the family $\mathscr{F}[C_{p^v}]$ is equal to
the family of subgroups $C_s \subset C_r$ for which $v_p(s) < v$.
Therefore, we may choose the universal space $E\mathscr{F}[C_{p^v}]$
to be a $C_r$-CW-complex that is non-equivariantly contractible and
that only has cells of orbit-type $C_r/C_{r/p^u}$. Indeed, in this
case we have
$$(E\mathscr{F}[C_{p^v}])^{C_s} = \begin{cases}
E\mathscr{F}[C_{p^v}] & \text{if $v_p(s) < v$} \cr
\varnothing & \text{if $v_p(s) \geqslant v$} \cr
\end{cases}$$
as required. We then have canonical isomorphisms
$$\begin{aligned}
{ [ S^q, S^{\lambda} \wedge E\mathscr{F}[C_{p^v}]_+ \wedge T(A) ]_{C_r} }
{} & \xleftarrow{\sim} [S^q, (S^{\lambda} \wedge E\mathscr{F}[C_{p^v}]_+ 
\wedge T(A))^{C_{r/p^u}} ]_{C_r/C_{r/p^u}} \cr
{} & \xleftarrow{\sim} [S^q, E\mathscr{F}[C_{p^v}]_+ 
\wedge (S^{\lambda} \wedge T(A))^{C_{r/p^u}} ]_{C_r/C_{r/p^u}} \cr
\end{aligned}$$
where for the second isomorphism we use that $E\mathscr{F}[C_{p^v}]$
is chosen to be $C_{r/p^u}$-fixed. The group isomorphism $\rho_{r/p^u}
\colon C_{p^u} \to C_r/C_{r/p^u}$ induces an isomorphism of categories
$\rho_{r/p^u}^*$ from the category of orthogonal $C_{p^u}$-spectra to
the category of orthogonal $C_r/C_{r/p^u}$-spectra. In particular,
this gives an isomorphism of the lower group above to the group
$$[S^q, \rho_{r/p^u}^*E\mathscr{F}[C_{p^v}]_+ \wedge
\rho_{r/p^u}^*(S^{\lambda} \wedge T(A))^{C_{r/p^u}} ]_{C_{p^u}} =
\mathbb{H}_q(C_{p^u},\TR_{\,\boldsymbol{\cdot}-\lambda}^{r/p^u}(A)).$$
This is indeed the desired Borel homology group, since
$\rho_{r/p^u}^*E\mathscr{F}[C_{p^v}]$ is a free $C_{p^u}$-CW-complex
which is non-equivariantly contractible.
\end{proof}

We recall that the Borel homology groups that appear in the statement
of Proposition~\ref{fundamentallongexactsequence} are the abutment of
the first quadrant skeleton spectral sequence
\begin{equation}\label{skeletonspectralsequence}
E_{s,t}^2 = H_s(C_{p^u},\TR_{t-\lambda}^{r/p^u}(A)) \Rightarrow
\mathbb{H}_{s+t}(C_{p^u},\TR_{\,\boldsymbol{\cdot}-\lambda}^{r/p^u}(A))
\end{equation}
from the group homology of $C_{p^u}$ with coefficients in the trivial
$C_{p^u}$-module\, $\TR_{t-\lambda}^{r/p^u}(A)$; see for
instance~\cite[Section~4]{h7}. 

We now specialize to the case $A = \Z$ and recall from
B\"{o}kstedt~\cite{bokstedt1} that $\TR_q^1(\Z)$ is zero if either
$q$ is negative or $q$ is positive and even, a free abelian group of
rank one if $q = 0$, and a finite cyclic group of order $i$ if $q =
2i-1$ is positive and odd; see  also~\cite{lindenstraussmadsen}.

\begin{lemma}\label{rational}Let $r$ be a positive integer, let $q$
be an integer, and let $\lambda$ be a finite dimensional complex
$\mathbb{T}$-representation. Then $\TR_{q-\lambda}^r(\Z)$ is a 
finitely generated abelian group whose rank is equal to the number
of positive divisors $e$ of $r$ for which
$q = 2\dim_{\C}(\lambda^{C_e})$. The group is zero for
$q < 2\dim_{\C}(\lambda^{C_r})$.
\end{lemma}

\begin{proof}Let $\ell(r,q,\lambda)$ denote the number of positive
divisors $e$ of $r$ with $q = 2\dim_{\C}(\lambda^{C_e})$ and note
that $\ell(r,q,\lambda)$ is zero, for $q$ odd. We prove by induction
on the number $k$ of prime divisors in $r$ that
$\TR_{q-\lambda}^r(\Z)$ is a finitely generated abelian group of rank 
$\ell(r,q,\lambda)$. If $k = 0$, or equivalently, if $r = 1$ the
statement follows from the result of B\"{o}kstedt which we recalled
above. Indeed, it follows from~\cite[Proposition~V.2.3]{mandellmay}
that, up to isomorphism,
$$\TR_{q-\lambda}^1(\Z) = \TR_{q - 2\dim_{\C}(\lambda)}^1(\Z).$$
So we let $k \geqslant 1$ and assume that the lemma has been proved,
for all $q$ and $\lambda$ as in the statement if $r$ has $k-1$ prime
divisors. Let $p$ be a prime divisor of $r$ and write $r = p^nr'$
with $r'$ not divisible by $p$. We consider the long exact sequence of
Proposition~\ref{fundamentallongexactsequence} with $u = n$ and $v =
1$,
$${ \cdots } \to
{ \mathbb{H}_q(C_{p^n},\TR_{\,\boldsymbol{\cdot}-\lambda}^{r'}(\Z)) } \xto{N_{p^n}}
{ \TR_{q-\lambda}^r(\Z) } \xto{R_p}
{ \TR_{q-\lambda'}^{r/p}(\Z) } \xto{\partial}
{ \mathbb{H}_{q-1}(C_{p^n},\TR_{\,\boldsymbol{\cdot}-\lambda}^{r'}(\Z)) } \to
{ \cdots }$$
Since $r'$ has only $k-1$ prime divisors, the inductive hypothesis
implies that in the skeleton spectral
sequence~(\ref{skeletonspectralsequence}), $E_{0,q}^2$ is a finitely
generated abelian group of rank $\ell(q,r',\lambda)$ and that the
groups $E_{s,t}^2$ with $s > 0$ are finite. Hence, the left-hand group
in the long exact sequence is finitely generated of rank
$\ell(q,r',\lambda)$. By further induction on $n \geqslant 0$, we may
assume that the group $\TR_{q-\lambda'}^{r/p}(\Z)$ is finitely
generated of rank $\ell(q,r/p,\lambda')$. The first part of the lemma
now follows from the formula
$$\ell(q,r',\lambda) + \ell(q,r/p,\lambda') = \ell(q,r,\lambda)$$
which holds since the two summands on the left-hand side count the
number of positive divisors $e$ of $r$ with $q = 2\dim_{\C}(\lambda^{C_e})$
for which $e$ is, respectively, prime to $p$ and divisible by $p$. The
second part of the lemma is proved in a similar manner.
\end{proof}

\begin{addendum}\label{stabilization} {\rm (i)} Let $m, r \geqslant 1$, $0
\leqslant \epsilon \leqslant 1$, and $i$ be integers, and let $d =
d(m,r)$ be the integer part of $(r-1)/m$. Then the canonical
projection induces an isomorphism
$$\lim_R \TR_{2i+\epsilon-\lambda_d}^r(\Z) \xto{\sim}
\TR_{2i+\epsilon-\lambda_d}^r(\Z),$$
provided that $m(i+1) < p^{v_p(r)+1}$ for every prime number $p$.

{\rm (ii)} Let $m \geqslant 1$, $0 \leqslant \epsilon \leqslant 1$,
and $i$ be integers, let $r \geqslant 1$ be an integer divisible by
$m$, and let $d = d(m,r)$. Then the canonical projection induces an
isomorphism
$$\lim_R \TR_{2i+\epsilon-\lambda_d}^{r/m}(\Z) \xto{\sim}
\TR_{2i+\epsilon-\lambda_d}^{r/m}(\Z)$$
provided that $i+1 < p^{v_p(r/m)+1}$ for every prime number $p$.
\end{addendum}

\begin{proof}We prove statement~(i); the proof of statement~(ii) is
similar. It suffices to show that for every prime number $p$ the 
restriction map
$$R_p \colon \TR_{q - \lambda_{d(m,pr)}}^{pr}(\Z) \to
\TR_{q-\lambda_{d(m,r)}}^r(\Z)$$
is an isomorphism if $q = 2i+\epsilon$ with $m(i+1) < p^{v_p(r)+1}$.
We write $r = p^{n-1}r'$ with $r'$ not divisible by $p$ and consider
the long exact sequence of
Proposition~\ref{fundamentallongexactsequence} with $u = n$ and $v =
1$,
$$\cdots \to
\mathbb{H}_q(C_{p^n},\TR_{\,\boldsymbol{\cdot}-\lambda_{d(m,pr)}}^{r'}(\Z)) \xto{N_{p^n}}
\TR_{q-\lambda_{d(m,pr)}}^{pr}(\Z) \xto{R_p}
\TR_{q-\lambda_{d(m,r)}}^r(\Z) \to
\cdots$$
The skeleton spectral sequence~(\ref{skeletonspectralsequence}) and
Lemma~\ref{rational} show that the left-hand group vanishes, provided
that $q < 2\dim_{\C}(\lambda_{d(m,pr)}^{C_{r'}})$. Therefore, the map
$R_p$ is an isomorphism for
$$i < \dim_{\C}(\lambda_{d(m,pr)}^{C_{r'}}) = \lfloor d(m,pr)/r'\rfloor.$$
We claim that $d(m,p^n) \leqslant \lfloor d(m,pr)/r'\rfloor$. Indeed,
this inequality is equivalent to the inequality $d(m,p^n) \leqslant
d(m,pr)/r'$ which is equivalent to the inequality $r'd(m,p^n)
\leqslant d(m,pr)$ which, in turn, is equivalent to the inequality
$r'd(m,p^n) \leqslant (p^nr'-1)/m$. We may rewrite this inequality as
$mr'd(m,p^n) \leqslant p^nr'-1$ or $mr'd(m,p^n) < p^nr'$ or $md(m,p^n)
< p^n$. But this inequality is equivalent to the inequality $md(m,p^n)
\leqslant p^n-1$ which, in turn, is equivalent to the inequality
$d(m,p^n) \leqslant (p^n-1)/m$ which holds. The claim
follows. Finally, a similar argument shows that the inequalities
$i < d(m,p^n)$ and $m(i+1) < p^n$ are equivalent.
\end{proof}

\section{The $p$-typical groups $\TR_{q-\lambda}^n(\Z_;p)$}\label{trsection} 

In this section, we prove Theorem~\ref{trthm} of the introduction. We
first show that after localization at the prime number $p$, the
groups $\TR_{q-\lambda}^r(A)$ decompose as products of the $p$-typical
groups 
$$\operatorname{TR}_{q-\lambda}^n(A;p) =
\operatorname{TR}_{q-\lambda}^{p^{n-1}}(A) = 
[ S^q \wedge (\mathbb{T}/C_{p^{n-1}})_+, S^{\lambda} \wedge T(A)
]_{\mathbb{T}}.$$

\begin{prop}\label{ptypicaldecomposition}Let $A$ be a unital
associative ring, $r \geqslant 1$ and $q$ integers, and $\lambda$ a
finite dimensional orthogonal $\mathbb{T}$-representation. Let $p$ be
a prime number and write $r = p^{n-1}r'$ with $r'$ not divisible by
$p$. Then the map
$$\gamma \colon \TR_{q-\lambda}^r(A) \to \textstyle{\prod_{j \mid r'}} \TR_{q-\lambda'}^n(A;p)$$
whose $j$th component is the composite map
$$\xymatrix{
{ \operatorname{TR}_{q-\lambda}^r(A) } \ar[r]^{F_j} &
{ \operatorname{TR}_{q-\lambda}^{r/j}(A) } \ar[r]^(.48){R_{r'/j}} &
{ \operatorname{TR}_{q-\lambda'}^{p^{n-1}}(A) } \ar@{=}[r]<-.2ex> &
{ \operatorname{TR}_{q-\lambda'}^{n^{\phantom{n}}}(A;p) } \cr
}$$
becomes an isomorphism after localization at $p$.
\end{prop}

\begin{proof}The proof is by induction on the number $k$ of positive
divisors of $r'$. If $k = 1$, or equivalently, if $r' = 1$ then
$\gamma$ is the identity map and the statement holds trivially. So we
let $k \geqslant 2$ and assume that the statement holds whenever $r'$
has at most $k-1$ divisors. Let $\ell$ be a prime divisor of $r'$,
and let $v = v_{\ell}(r') = v_{\ell}(r)$. We show that the map
$$(R_{\ell},F_{\ell^v}) \colon \TR_{q-\lambda}^r(A) \to
\TR_{q-\lambda'}^{r/\ell}(A) \times \TR_{q-\lambda}^{r/\ell^v}(A)$$
becomes an isomorphism after localization at $p$. This will prove the
induction step, since $r/\ell$ and $r/\ell^v$ have at most $k-1$
divisors and
$\gamma = (\gamma \times \gamma) \circ (R_{\ell},F_{\ell^v})$. Now, by
Proposition~\ref{fundamentallongexactsequence}, we have the long exact
sequence 
$$\cdots \to
\mathbb{H}_q(C_{\ell^v},
\TR_{\,\boldsymbol{\cdot}-\lambda}^{r/\ell^v}(A)) \xto{N_{\ell^v}}
\TR_{q-\lambda}^r(A) \xto{R_{\ell}}
\TR_{q-\lambda'}^{r/\ell}(A) \xto{\partial}
\mathbb{H}_{q-1}(C_{\ell^v},
\TR_{\,\boldsymbol{\cdot}-\lambda}^{r/\ell^v}(A)) \to \cdots$$
Moreover, one readily shows that the composite map
$$\TR_{q-\lambda}^{r/\ell^v}(A) \xto{\epsilon}
\mathbb{H}_q(C_{\ell^v},\TR_{\,\boldsymbol{\cdot}-\lambda}^{r/\ell^v}(A))
\xto{N_{\ell^v}}
\TR_{q-\lambda}^r(A) \xto{F_{\ell^v}}
\TR_{q-\lambda}^{r/\ell^v}(A),
$$
where $\epsilon$ is the edge homomorphism of the skeleton spectral
sequence~(\ref{skeletonspectralsequence}), is equal to the composition
$F_{\ell^v}V_{\ell^v}$ of the Frobenius and Verschiebung maps which,
in turn, is equal to the map given by multiplication by
$\ell^v$. Hence, after localization at $p$, $F_{\ell^v}$ is the
projection onto a direct summand of the group
$\TR_{q-\lambda}^r(Z)$. The long exact sequence shows that the map
$(R_{\ell},F_{\ell^v})$ becomes an isomorphism after localization at
$p$ as desired.
\end{proof}

The maps $F_s$, $V_s$, and $R_s$ may also be expressed as products of
their $p$-typical analogs
$$\begin{aligned}
F = F_p \colon & \operatorname{TR}_{q-\lambda}^n(A;p) \to
\operatorname{TR}_{q-\lambda}^{n-1}(A;p) 
\hskip7mm \text{(Frobenius)} \hfill\space \cr
V = V_p \colon & \operatorname{TR}_{q-\lambda}^{n-1}(A;p) \to
\operatorname{TR}_{q-\lambda}^n(A;p) 
\hskip7mm \text{(Verschiebung)} \hfill\space \cr
R = R_p \colon & \operatorname{TR}_{q-\lambda}^n(A;p) \to
\operatorname{TR}_{q-\lambda'}^{n-1}(A;p) \hskip6.3mm
\text{(restriction)} \hfill\space \cr
\end{aligned}$$
Suppose that $r = st$ and write $s = p^vs'$ and $t = p^{n-v-1}t'$ with
$s'$ and $t'$ not  divisible by $p$. Then there are three commutative
square diagrams
\begin{equation}\label{frobeniusverschiebungrestriction}
\xymatrix{
{ \operatorname{TR}_{q-\lambda}^r(A) } \ar[r]^(.4){\gamma}
\ar[d]<-1ex>_{F_s} & 
{ \prod_{j \mid r'} \operatorname{TR}_{q-\lambda'}^n(A;p) }
\ar[d]<-1ex>_{F_s^{\gamma}} &
{ \operatorname{TR}_{q-\lambda}^r(A) } \ar[r]^(.4){\gamma}
\ar[d]^{R_s} & 
{ \prod_{j \mid r'} \operatorname{TR}_{q-\lambda'}^n(A;p) }
\ar[d]^{R_s^{\gamma}} \cr
{ \operatorname{TR}_{q-\lambda}^t(A) } \ar[r]^(.4){\gamma}
\ar[u]<-1ex>_{V_s} & 
{ \prod_{j \mid t'} \operatorname{TR}_{q-\lambda'}^{n-v}(A;p) }
\ar[u]<-1ex>_{V_s^{\gamma}} & 
{ \operatorname{TR}_{q-\lambda'}^t(A) } \ar[r]^(.4){\gamma} & 
{ \prod_{j \mid t'} \operatorname{TR}_{q-\lambda''}^{n-v}(A;p), } \cr
}
\end{equation}
where the maps $F_s^{\gamma}$, $V_s^{\gamma}$, and $R_s^{\gamma}$ are
defined as follows: The map $F_s^{\gamma}$ takes the factor indexed by
a divisor $j$ of $r'$ that is divisible by $s'$ to the factor indexed
by the divisor $j/s'$ of $t'$ by the map $F^v$ and annihilates the
remaining factors. The map $V_s^{\gamma}$ takes the factor indexed by
the divisor $j$ of $t'$ to the factor indexed by the divisor $s'j$ of
$r'$ by the map $s'V^v$. Finally, the map $R_s^{\gamma}$ takes the
factor indexed by a divisor $j$ of $t'$ to the factor indexed by the
same divisor $j$ of $t'$ by the map $R^v$ and annihilates the factors
indexed by divisors $j$ of $r'$ that do not divide $t'$.

Let $M_p$ be the equivariant Moore spectrum defined by the mapping
cone of the multiplication by $p$ map on the sphere
$\mathbb{T}$-spectrum. The equivariant homotopy groups with
$\Z/p\Z$-coefficients
$$\TR_{q-\lambda}^n(\Z;p,\Z/p\Z) = 
[S^q \wedge (\mathbb{T}/C_{p^{n-1}})_+, M_p \wedge
S^{\lambda} \wedge T(\Z)]_{\mathbb{T}}$$
were evaluated for $p$ odd by Tsalidis~\cite{tsalidis1}, and for all
$p$ by the first and second authors~\cite{angeltveitgerhardt}. We
recall the result. 

\begin{theorem}[(Angeltveit-Gerhardt, Tsalidis)]\label{modp}Let $p$ be
a prime number, let $n$ be a positive integer, let $\lambda$ be a
finite dimensional complex $\mathbb{T}$-representation, and define
$$\delta_p(\lambda) = 
(1-p) \sum_{s \geqslant 0} \dim_{\C}(\lambda^{C_{p^s}}) p^s.$$
Then the finite $\Z_{(p)}$-modules $\TR_{q-\lambda}^n(\Z;p,\Z/p\Z)$
have the following structure:
\begin{enumerate}
\item[(i)] For $q \geqslant 2 \dim_{\C}(\lambda)$,
$\TR_{q-\lambda}^n(\Z;p,\Z/p\Z)$ has length $n$, if $q$ is congruent
to $2\delta_p(\lambda)$ or $2\delta_p(\lambda)-1$ modulo $2p^n$, and
$n-1$, otherwise.
\item[(ii)] For $2\dim_{\C}(\lambda^{C_{p^s}}) \leqslant q <
2\dim_{\C}(\lambda^{C_{p^{s-1}}})$ with $1 \leqslant s < n$,
$\TR_{q-\lambda}^n(\Z;p,\Z/p\Z)$ has length $n-s$, if $q$ is congruent 
to $2\delta_p(\lambda^{C_{p^s}})$ or $2\delta_p(\lambda^{C_{p^s}})-1$ modulo
$2p^{n-s}$, and $n-s-1$, otherwise.
\item[(iii)] For $q < 2\dim_{\C}(\lambda^{C_{p^{n-1}}})$,
$\TR_{q-\lambda}^n(\Z;p,\Z/p\Z)$ is zero.
\end{enumerate}
\end{theorem}

We show in Corollary~\ref{mod2} below that the groups
$\TR_{q-\lambda}^n(\Z;p,\Z/p\Z)$ have exponent $p$ for all prime
numbers $p$.

\begin{proof}[of Theorem~\ref{trthm}~(i)]By Lemma~\ref{rational},
$\TR_{2i-\lambda}^n(\Z;p)$ is a finitely generated abelian group, and
hence, it suffices to show that it is torsion free. We first show that
the $p$-torsion subgroup is trivial. Comparing Theorem~\ref{modp} and
Lemma~\ref{rational}, we find 
that for all integers $i$
$$\length_{\Z_{(p)}}\!\!\TR_{2i-\lambda}^n(\Z;p,\Z/p\Z)
- \length_{\Z_{(p)}}\!\!\TR_{2i-1-\lambda}^n(\Z;p,\Z/p\Z)
= \operatorname{rk}_{\Z} \TR_{2i-\lambda}^n(\Z;p).$$
Moreover, Lemma~\ref{rational} shows that for every integer $i$,
$\TR_{2i-1-\lambda}^n(\Z;p)_{(p)}$ is a finite $p$-primary torsion
group. By a Bockstein spectral sequence argument, we conclude that
$\TR_{2i-\lambda}^n(\Z;p)_{(p)}$ is torsion free;
compare~\cite[Proposition~13]{h7}. This shows that the group
$\TR_{2i-\lambda}^n(\Z;p)$ has no $p$-torsion. To see that it has no
prime to $p$ torsion, we use that, by
Proposition~\ref{ptypicaldecomposition}, the map
$$(R^{n-1-s}F^s) \colon \TR_{2i-\lambda}^n(\Z;p) \to
\prod_{0 \leqslant s < n} \TR_{2i-\lambda^{(n-1-s)}}^1(\Z;p)$$
becomes an isomorphism after inverting $p$. Therefore, B\"{o}kstedt's
result recalled earlier shows that also $\TR_{2i-\lambda}^n(\Z;p)[1/p]$
is torsion free. This completes the proof.
\end{proof}

\begin{lemma}\label{evensurjective}Let $p$ be a prime number, let $n$
be a positive integer, and let $\lambda$ be a finite dimensional
complex $\mathbb{T}$-representation. Then the restriction map
$$R \colon \TR_{q-\lambda}^n(\Z;p) \to
\TR_{q-\lambda'}^{n-1}(\Z;p)$$
is surjective for every even integer $q$.
\end{lemma}

\begin{proof}We see as in the proof of Addendum~\ref{stabilization}
that the map of the statement is an epimorphism, for $q \leqslant 2
\dim_{\C}(\lambda)$. Moreover, Lemma~\ref{rational} and
Theorem~\ref{trthm}(i) show that $\TR_{q-\lambda'}^{n-1}(\Z;p)$ is 
zero, for $q > 2 \dim_{\C}(\lambda')$ and even. The lemma follows,
since $\dim_{\C}(\lambda') \leqslant \dim_{\C}(\lambda)$.
\end{proof}

\begin{prop}\label{oddsupport}Let $p$ be a prime number, let $n$ be a
positive integer, and let $\lambda$ be a finite dimensional complex
$\mathbb{T}$-representation. Then in the skeleton spectral sequence
$$E_{s,t}^2 = H_s(C_{p^{n-1}},\TR_{t-\lambda}^1(\Z;p))
\Rightarrow
\mathbb{H}_{s+t}(C_{p^{n-1}},\TR_{\,\boldsymbol{\cdot}-\lambda}^1(\Z;p)),$$ 
every non-zero differential $d^r \colon E_{s,t}^r \to E_{s-r,t+r-1}^r$
is supported in odd total degree.
\end{prop}

\begin{proof}We must show that if $s + t$ is even then
$d^r \colon E_{s,t}^r \to E_{s-r,t+r-1}^r$ is zero. Suppose first that
$s$ and $t$ are both even. Then $E_{s,t}^2$ is zero unless $s = 0$
and $t = 2\dim_{\C}(\lambda)$. Therefore, in this case, $d^r \colon
E_{s,t}^r \to E_{s-r,t+r-1}^r$ is zero. Suppose next that $s$ and
$t$ are both odd and that $r$ is even. Then $E_{s,t}^2$ is zero
unless $t > 2\dim_{\C}(\lambda)$, and $E_{s-r,t+r-1}^2$ is zero
unless $t+r-1 = 2\dim_{\C}(\lambda)$. It follows that, also in this
case, $d^r \colon E_{s,t}^r \to E_{s-r,t+r-1}^r$ is zero. It remains
to prove that if $r$, $s$, and $t$ are all odd then $d^r \colon
E_{s,t}^r \to E_{s-r,t+r-1}^r$ is zero.

To this end, we use that, for all integers $n' \geqslant n
\geqslant 1$, the iterated Frobenius map 
$$F^{n'-n} \colon \mathbb{H}_q(C_{p^{n'-1}},
\TR_{\,\boldsymbol{\cdot}-\lambda}^1(\Z;p))
\to \mathbb{H}_q(C_{p^{n-1}}, \TR_{\,\boldsymbol{\cdot}-\lambda}^1(\Z;p))$$
induces a map of skeleton spectral sequences that we write
$$F^{n'-n} \colon E_{s,t}^r(n',\lambda) \to E_{s,t}^r(n,\lambda).$$
The map of $E^2$-terms is given by the transfer map in group homology
and is readily evaluated; see for instance~\cite[Lemma~6]{h7}. It is
surjective if $s$ is odd, and zero if is $s$ even, $t$ is odd, and $n'-n$
is sufficiently large. We  prove by induction on $r \geqslant 2$ that,
for all odd integers $s$ and $t$ and for all $n$ and $\lambda$ as in
the statement, the differential $d^r \colon E_{s,t}^r(n,\lambda) \to
E_{s-r,t+r-1}^r(n,\lambda)$ is zero and the Frobenius map $F \colon
E_{s,t}^r(n+1,\lambda) \to E_{s,t}^r(n,\lambda)$ surjective. The case
$r = 2$ was proved above. So we let $r \geqslant 3$ and assume,
inductively, that the statement has been proved for $r-1$. Since the
differential $d^{r-1} \colon E_{s,t}^{r-1}(n,\lambda) \to
E_{s-(r-1),t+r-2}^{r-1}(n,\lambda)$ is zero by the inductive
hypothesis, the canonical isomorphism
$$H( E_{s+r-1,t-(r-2)}^{r-1}(n,\lambda) \xto{d^{r-1}} 
E_{s,t}^{r-1}(n,\lambda) \xto{d^{r-1}} 
E_{s-(r-1),t+r-2}^{r-1}(n,\lambda) ) \xto{\sim} 
E_{s,t}^r(n,\lambda)$$
gives rise to a canonical surjection $\pi \colon
E_{s,t}^{r-1}(n,\lambda) \twoheadrightarrow
E_{s,t}^r(n,\lambda)$. Moreover, by naturality of the skeleton
spectral sequence, the diagram
$$\xymatrix{
{ E_{s,t}^{r-1}(n+1,\lambda) } \ar@{->>}[r]^{\pi} \ar[d]^{F} &
{ E_{s,t}^r(n+1,\lambda) } \ar[d]^{F} \cr
{ E_{s,t}^{r-1}(n,\lambda) } \ar@{->>}[r]^{\pi} &
{ E_{s,t}^r(n,\lambda) } \cr
}
$$
commutes. Since the left-hand vertical map $F$ is surjective by the
inductive hypothesis, we conclude that the right-hand vertical map $F$
is surjective as desired. It remains to prove that the differential
$d^r \colon E_{s,t}^r(n,\lambda) \to E_{s-r,t+r-1}^r(n,\lambda)$ is
zero. The case where $r$ is even was proved above, and in the case
where $r$ is odd, we consider the following commutative diagram:
$$\xymatrix{
{ E_{s,t}^r(n',\lambda) } \ar[r]^(.42){d^r} \ar[d]^{F^{n'-n}} &
{ E_{s-r,t+r-1}^r(n',\lambda) } \ar[d]^{F^{n'-n}} \cr
{ E_{s,t}^r(n,\lambda) } \ar[r]^(.42){d^r} &
{ E_{s-r,t+r-1}^r(n,\lambda) } \cr
}$$
We have just proved that the left-hand vertical map is
surjective. Moreover, as recalled above, the right-hand vertical map 
is zero if $n' - n$ is sufficiently large. Indeed, $s-r$ is even and
$t-r+1$ is odd. Hence, the lower horizontal map $d^r$ is zero as
desired. This proves the induction step and the proposition.
\end{proof}

\begin{remark}The proof of Proposition~\ref{oddsupport} also shows
that in the Tate spectral sequence
$$\hat{E}_{s,t}^2 = \hat{H}^{-s}(C_{p^{n-1}},\TR_{t-\lambda}^1(\Z;p))
\Rightarrow 
\hat{\mathbb{H}}^{-s-t}(C_{p^{n-1}},
\TR_{\,\boldsymbol{\cdot}-\lambda}^1(\Z;p)),$$ 
every non-zero differential is supported in even total degree. It
remains an important problem to determine the differential structure
of the two spectral sequences.
\end{remark}

\begin{proof}[of Theorem~\ref{trthm}~(ii)]First, for $i \leqslant
\dim_{\C}(\lambda)$, we see as in the proof of
Lemma~\ref{evensurjective} that the restriction map 
$$R \colon \TR_{q-\lambda}^n(\Z;p) \to \TR_{q-\lambda'}^{n-1}(\Z;p)$$
is an isomorphism, so the statement holds in this case. Next, for
$i = \dim_{\C}(\lambda)+1$,
Proposition~\ref{fundamentallongexactsequence} and
Lemmas~\ref{rational} and~\ref{evensurjective} give a short exact
sequence
$$0 \to 
\mathbb{H}_q(C_{p^{n-1}}, \TR_{\,\boldsymbol{\cdot}-\lambda}^1(\Z;p)) \to 
\TR_{q-\lambda}^n(\Z;p) \to \TR_{q-\lambda'}^{n-1}(\Z;p) \to 0$$
and the skeleton spectral sequence~(\ref{skeletonspectralsequence})
shows that the left-hand group has order $p^{n-1}$. So the statement
also holds in this case. Finally, for $i > \dim_{\C}(\lambda)+1$,
Proposition~\ref{fundamentallongexactsequence} and
Lemma~\ref{rational} give a four-term exact sequence
$$\begin{aligned}
0 & \to
\mathbb{H}_q(C_{p^{n-1}}, \TR_{\,\boldsymbol{\cdot}-\lambda}^1(\Z;p)) \to
\TR_{q-\lambda}^n(\Z;p) \to
\TR_{q-\lambda'}^{n-1}(\Z;p) \cr
{} & \to
\mathbb{H}_{q-1}(C_{p^{n-1}}, \TR_{\,\boldsymbol{\cdot}-\lambda}^1(\Z;p)) \to
0 \cr
\end{aligned}$$
which shows that the orders of the four groups satisfy the equality
$$\lvert \TR_{q-\lambda}^n(\Z;p) \rvert \big/ \lvert
\TR_{q-1-\lambda'}^{n-1}(\Z;p) \rvert = 
\lvert \mathbb{H}_q(C_{p^{n-1}},
\TR_{\,\boldsymbol{\cdot}-\lambda}^1(\Z;p)) \rvert \big/ \lvert
\mathbb{H}_{q-1}(C_{p^{n-1}},
\TR_{\,\boldsymbol{\cdot}-\lambda}^1(\Z;p)) \rvert.$$ 
To evaluate the ratio on the right-hand side, we consider the skeleton
spectral sequence~(\ref{skeletonspectralsequence}). We may write the
ratio in question as
$$\lvert \mathbb{H}_q(C_{p^{n-1}},
\TR_{\,\boldsymbol{\cdot}-\lambda}^1(\Z;p)) \rvert \big/ 
\lvert \mathbb{H}_{q-1}(C_{p^{n-1}},
\TR_{\,\boldsymbol{\cdot}-\lambda}^1(\Z;p)) \rvert =  
\big( \! \prod_{s+t = q} \lvert E_{s,t}^{\infty} \rvert \big) \big/
\big( \!\! \prod_{s+t = q-1} \lvert E_{s,t}^{\infty} \rvert \big).$$
By Proposition~\ref{oddsupport}, for all $r \geqslant 2$
and all $s$ and $t$ with $s+t$ is odd we have an exact sequence
$$0 \to E_{s,t}^{r+1} \to E_{s,t}^r \xto{d^r} E_{s-r,t+r-1}^r \to
E_{s-r,t+r-1}^{r+1} \to 0.$$
Hence, by induction on $r$, we find that
$$\big( \! \prod_{s+t = q} \lvert E_{s,t}^{\infty} \rvert \big) \big/
\big( \!\! \prod_{s+t = q-1} \lvert E_{s,t}^{\infty} \rvert \big) 
= \big( \! \prod_{s+t = q} \lvert E_{s,t}^2 \rvert \big) \big/
\big( \!\! \prod_{s+t = q-1} \lvert E_{s,t}^2 \rvert \big)$$
and the ratio on the right-hand side is readily seen to be equal to
$$\lvert E_{0,q}^2 \rvert \cdot \lvert
E_{q-2\dim_{\C}(\lambda),2\dim_{\C}(\lambda)}^2 \rvert 
= (i-\dim_{\C}(\lambda)) \cdot p^{n-1}.$$
This completes the proof.
\end{proof}

\begin{proof}[of Theorem~\ref{trthm}~{\rm(iii)}]First, for $q$ odd, we
use that the Verschiebung map in question is equal to the
composition of the edge homomorphism 
$$\epsilon \colon \TR_{q-\lambda}^{n-1}(\Z;p) \to
\mathbb{H}_q(C_p,\TR_{\,\boldsymbol{\cdot}-\lambda}^{n-1}(\Z;p))$$
of the skeleton spectral sequence~(\ref{skeletonspectralsequence}) and
the norm map $N_{n-1}$ in the long exact sequence
$$\xymatrix{
{ \cdots } \ar[r] &
{ \mathbb{H}_q(C_p,\TR_{\boldsymbol{\cdot}-\lambda}^{n-1}(\Z;p)) }
\ar[r]^(.58){N_{n-1}} &
{ \TR_{q-\lambda}^n(\Z;p) } \ar[r]^(.45){R^{n-1}} &
{ \TR_{q-\lambda^{(n-1)}}^1(\Z;p) } \ar[r] &
{ \cdots } \cr
}$$
from Proposition~\ref{fundamentallongexactsequence}. Since $q$ is odd,
Lemma~\ref{evensurjective} shows that the latter map is
injective. Hence, it will suffice to show that also the edge
homomorphism is injective, or equivalently, that in the skeleton
spectral sequence, all differentials of the form
$$d^r \colon E_{r,q-r+1}^r \to E_{0,q}^r$$
are zero. If $r$ is even, then $q-r+1$ is even and
Theorem~\ref{trthm}~(i) shows that the group $E_{r,q-r+1}^r$ is
zero. Hence, $d^r$ is zero in this case. If $r$ is odd, we consider
the iterated Frobenius map
$$F^{v'-v} \colon
\mathbb{H}_q(C_{p^{v'}},\TR_{\boldsymbol{\cdot}-\lambda}^{n-1}(\Z;p))
\to \mathbb{H}_q(C_{p^v},\TR_{\boldsymbol{\cdot}-\lambda}^{n-1}(\Z;p)).$$
It induces a map of spectral sequences that we write
$$F^{v'-v} \colon E_{s,t}^r(v',\lambda) \to E_{s,t}^r(v,\lambda).$$
As in the proof of Theorem~\ref{trthm}~(ii), an induction on $r
\geqslant 2$ shows that, for all $v' \geqslant v \geqslant 1$, the
left-hand vertical map and the horizontal maps in the diagram 
$$\xymatrix{
{ E_{r,q-r+1}^r(v',\lambda) } \ar[r]^(.55){d^r} \ar[d]^{F^{v'-v}} &
{ E_{0,q}^r(v',\lambda) } \ar[d]^{F^{v'-v}} \cr
{ E_{r,q-r+1}^r(v,\lambda) } \ar[r]^(.55){d^r} &
{ E_{0,q}^r(v,\lambda) } \cr
}$$
are surjective and zero, respectively. The proof of the induction
step uses that, for $v' - v$ sufficiently large, the right-hand
vertical map is zero. This proves the statement for $q$ odd.  

Finally, suppose that $q$ is even. We let $\mu$ be the
direct sum of $\lambda$ and the one dimensional complex representation
$\C(p^{n-2})$. Then~\cite[Proposition~4.2]{h6} gives a
long exact sequence
$$\cdots \to \TR_{q+1-\mu}^n(\Z;p) \xto{(F,-Fd)}
\overset{ \displaystyle{ \TR_{q-1-\lambda}^{n-1}(\Z;p) }}{ \underset{
\displaystyle{ \TR_{q-\lambda}^{n-1}(\Z;p) }}{ \oplus }} \xto{dV+V}
\TR_{q-\lambda}^n(\Z;p) \xto{\iota_*}
\TR_{q-\mu}^n(\Z;p) \to \cdots$$
Now, we proved in Thm~\ref{trthm}~(i) that $\TR_{q+1-\mu}^n(\Z;p)$ and
$\TR_{q-1-\lambda}^{n-1}(\Z;p)$ are finite abelian groups while
$\TR_{q-\lambda}^{n-1}(\Z;p)$, $\TR_{q-\lambda}^n(\Z;p)$, and
$\TR_{q-\mu}^n(\Z;p)$ are free abelian groups. Hence, we obtain the
exact sequence of free abelian groups
$$\xymatrix{
{ 0 } \ar[r] &
{ \TR_{q-\lambda}^{n-1}(\Z;p) } \ar[r]^{V} &
{ \TR_{q-\lambda}^n(\Z;p) } \ar[r]^{\iota_*} &
{ \TR_{q-\mu}^n(\Z;p) } \cr
}$$
which shows that for $q$ even, the Verschiebung map is the inclusion
of a direct summand as stated. This concludes the proof of
Theorem~\ref{trthm}.
\end{proof}

\begin{cor}\label{mod2}Let $n$ be a positive integer, $p$ a prime
number, and $\lambda$ a finite dimensional complex
$\mathbb{T}$-representation. Then $\TR_{q-\lambda}^n(\Z;p,\Z/p\Z)$ has
exponent $p$ for every integer $q$. 
\end{cor}

\begin{proof}We first let $p = 2$ and consider the coefficient
long exact sequence
$$\cdots \to \TR_{q-\lambda}^n(\Z;2) \xto{2} 
\TR_{q-\lambda}^n(\Z;2) \xto{\iota} 
\TR_{q-\lambda}^n(\Z;2,\Z/2\Z) \xto{\beta}
\TR_{q-1-\lambda}^n(\Z;2) \to \cdots.$$
It is proved in~\cite[Theorem~1.1]{arakitoda} that the composition
$$\TR_{q-\lambda}^n(\Z;2,\Z/2\Z) \xto{\beta} \TR_{q-1-\lambda}^n(\Z;2)
\xto{\eta} \TR_{q-\lambda}^n(\Z;2) \xto{\iota}
\TR_{q-\lambda}^n(\Z;2,\Z/2\Z)$$
is equal to multiplication by $2$. Now, Theorem~\ref{trthm} shows
that for $q$ odd the map $\beta$ is zero, and that for $q$ even the
map $\eta$ is zero. Hence, the group $\TR_{q-\lambda}^n(\Z;2,\Z/2\Z)$ is
annihilated by multiplication by $2$ as stated. Finally, for $p$
odd,~\cite[Theorem~1.1]{arakitoda} shows that the multiplication by
$p$ map on $\TR_{q-\lambda}^n(\Z;p,\Z/p\Z)$ is equal to zero.
\end{proof}

\section{The groups $K_q(\Z[x]/(x^m),(x))$}\label{ksection}

In this section, we prove Theorem~\ref{kthm} of the introduction. 

\begin{prop}\label{middleterm}Let $m$ and $r$ be positive integers,
let $i$ be a non-negative integer, and let $d = d(m,r)$ be the integer
part of $(r-1)/m$. Then:
\begin{enumerate}
\item[(i)]The abelian group $\lim_R \TR_{2i-\lambda_d}^r(\Z)$ is free
of rank $m$.
\item[(ii)]The abelian group $\lim_R \TR_{2i-1-\lambda_d}^r(\Z)$ is
finite of order $(mi)!(i!)^m$. 
\end{enumerate}
\end{prop}

\begin{proof}It follows from Lemma~\ref{rational} and
Addendum~\ref{stabilization} that the groups $\lim_R
\TR_{q-\lambda_d}^r(\Z)$ are finitely generated. Hence, it suffices to
show that for every prime number $p$, the $\Z_{(p)}$-module $\lim_R 
\TR_{2i-\lambda_d}^r(\Z)_{(p)}$ has finite rank $m$ and the
$\Z_{(p)}$-module $\lim_R \TR_{2i-1-\lambda_d}^r(\Z)_{(p)}$ has finite
length $v_p((mi)!(i!)^m)$. We fix a prime number $p$ and let $I_p$ be
the set of positive integers not divisible by $p$. It follows from
Proposition~\ref{ptypicaldecomposition} that there is a canonical
isomorphism
$$\lim_R \TR_{q-\lambda_d}^r(\Z)_{(p)} \xto{\sim}
\prod_{j \in I_p} \lim_R \TR_{q-\lambda_d}^n(\Z;p)_{(p)},$$
where, on the left-hand side, the limit ranges over the set of
positive integers ordered under division and $d = d(m,r)$, and where,
on the right-hand side, the limits range over the set of non-negative
integers ordered additively and $d = d(m,p^{n-1}j)$. Moreover, on the
$j$th factor of the  product, the canonical projection
$$\lim_R \TR_{q-\lambda_d}^n(\Z;p)_{(p)} \to 
\TR_{q-\lambda_d}^s(\Z;p)_{(p)}$$
is an isomorphism for $q < 2d(m,p^sj)$; see~\cite[Lemma~2.6]{h6}. The
requirement that $2i-2$ and $2i-1$ be strictly smaller than
$2d(m,p^sj)$ is equivalent to the requirement that $mi < p^sj$. Hence,
for $q = 2i-2$ or $q = 2i-1$, we have a canonical isomorphism
$$\lim_R \TR_{q-\lambda_d}^r(\Z)_{(p)} \xto{\sim}
\prod_{\substack{1 \leqslant j \leqslant mi \\ j \in I_p}}
\TR_{q-\lambda_d}^s(\Z;p)_{(p)},$$
where $s = s_p(m,i,j)$ is the unique integer such that
$$p^{s-1}j \leqslant mi < p^sj.$$
Now, from Theorem~\ref{trthm}~(i) we find that
$$\begin{aligned}
{} & \operatorname{rk}_{\Z_{(p)}} \lim_R \TR_{2i-\lambda_d}^r(\Z)_{(p)} 
= \lvert \{ j \in I_p \mid \text{$i = d(m,p^{n-1}j)$, for some 
$n \geqslant 1$} \} \rvert \cr 
{} & = \lvert \{ r \in \N \mid i = d(m,r) \} \rvert
 = \lvert \{mi+1, mi+2, \dots, mi+m\} \rvert = m \cr
\end{aligned}$$
which proves statement~(i). Similarly, we have
$$\length_{\Z_{(p)}} \lim_R \TR_{2i-1-\lambda_d}^r(\Z)_{(p)} 
= \sum_{\substack{1 \leqslant j \leqslant mi\\ j \in I_p}}
\length_{\Z_{(p)}} \TR_{2i-1-\lambda_d}^s(\Z;p)_{(p)},$$
where $s = s_p(m,i,j)$, and Theorem~\ref{trthm}~(ii) shows that the
right-hand side is equal to
$$\begin{aligned}
{} & \sum_{\substack{1 \leqslant j \leqslant mi \\ j \in I_p}}
\sum_{1 \leqslant t \leqslant s} \big( v_p(i-d(m,p^{t-1}j)) + t-1 \big) 
= \sum_{1 \leqslant k \leqslant mi} \big( v_p(i-d(m,k)) + v_p(k) \big) \cr
{} & = m \sum_{0 \leqslant l < i} v_p(i-l) + \!\! \sum_{1 \leqslant k
  \leqslant mi} v_p(k) 
= m \sum_{1 \leqslant k \leqslant i} v_p(k) + \!\! \sum_{1 \leqslant k
  \leqslant mi} v_p(k) = v_p((i!)^m(mi)!) \cr
\end{aligned}$$
which proves statement~(ii). 
\end{proof}

\begin{prop}\label{lefthandterm}Let $m$ and $r$ be positive integers,
let $i$ be a non-negative integer, and let $d = d(m,r)$ be the integer
part of $(r-1)/m$. Then:
\begin{enumerate}
\item[(i)]The abelian group $\lim_R \TR_{2i-\lambda_d}^{r/m}(\Z)$ is
free of rank $1$.
\item[(ii)]The abelian group $\lim_R \TR_{2i-1-\lambda_d}^{r/m}(\Z)$ is
finite of order $(i!)^2$. 
\end{enumerate}
\end{prop}

\begin{proof}It follows from Lemma~\ref{rational} and
Addendum~\ref{stabilization} that the groups $\lim_R
\TR_{q-\lambda_d}^{r/m}(\Z)$ are finitely generated. We fix a prime
number $p$ and write $m = p^vm'$ with $m'$ not divisible by $p$. Then
for $q = 2i-2$ and $q = 2i-1$, there is a canonical isomorphism 
$$\lim_R \TR_{q-\lambda_d}^{r/m}(\Z)_{(p)} \xto{\sim}
\prod_{\substack{1 \leqslant j \leqslant mi \\ j \in m'I_p}}
\TR_{q-\lambda_d}^{s-v}(\Z;p)_{(p)},$$
where $s = s_p(m,i,j)$. From Theorem~\ref{trthm}~(i) we find
$$\begin{aligned}
{} & \operatorname{rk}_{\Z_{(p)}} \lim_R
\TR_{2i-\lambda_d}^{r/m}(\Z)_{(p)}  
= \lvert \{ j \in m'I_p \mid \text{$i = d(m,p^{n+v-1}j)$, for some
$n \geqslant 1$} \} \rvert \cr 
{} & = \lvert \{ r \in m\N \mid i = d(m,r) \} \rvert
 = \lvert \{mi+m\} \rvert = 1 \cr
\end{aligned}$$
which proves statement~(i). Similarly, from Theorem~\ref{trthm}~(ii), we have
$$\begin{aligned}
{} & \length_{\Z_{(p)}} \lim_R \TR_{2i-1-\lambda_d}^r(\Z)_{(p)} 
= \sum_{\substack{1 \leqslant j \leqslant mi\\ j \in m'I_p}}
\length_{\Z_{(p)}} \TR_{2i-1-\lambda_d}^{s-v}(\Z;p)_{(p)}  \cr
{} & = \sum_{\substack{1 \leqslant j \leqslant mi \\ j \in m'I_p}}
\sum_{1 \leqslant t \leqslant s-v} \big( v_p(i-d(m,p^{t+v-1}j)) + t-1 \big) 
= \sum_{1 \leqslant k \leqslant i} \big( v_p(i-d(m,km)) + v_p(k) \big) \cr
{} & = \sum_{0 \leqslant l < i} v_p(i-l) + \!\! \sum_{1 \leqslant k
  \leqslant i} v_p(k) 
= 2 \sum_{1 \leqslant k \leqslant i} v_p(k) = v_p((i!)^2) \cr
\end{aligned}$$
which proves statement~(ii). 
\end{proof}

\begin{prop}\label{verschiebung}Let $m$ and $r$ be positive integers,
and let $d = d(m,r)$ be the integer part of $(r-1)/m$. Then the
Verschiebung map
$$V_m \colon \lim_R \TR_{q-\lambda_d}^{r/m}(\Z) \to
\lim_R \TR_{q-\lambda_d}^r(\Z)$$
is injective for all integers $q$, and has free abelian cokernel for
all even integers $q$.
\end{prop}

\begin{proof}We fix a prime number $p$ and show that the Verschiebung
map
$$V_m \colon \lim_R \TR_{q-\lambda_d}^{r/m}(\Z)_{(p)} \to
\lim_R \TR_{q-\lambda_d}^r(\Z)_{(p)}$$
is injective for all integers $q$, and has cokernel a free 
$\Z_{(p)}$-module for all even integers $q$. We write $m = p^vm'$
with $m'$ not divisible by $p$. Then for $q = 2i-2$ and $q = 2i-1$ the
map $V_m$ is canonically isomorphic to the map
$$m'V^v \colon 
\prod_{\substack{1 \leqslant j \leqslant mi \\ j \in m'I_p}}
\TR_{q-\lambda_d}^{s-v}(\Z;p)_{(p)} \to
\prod_{\substack{1 \leqslant j \leqslant mi \\ j \in I_p}}
\TR_{q-\lambda_d}^s(\Z;p)_{(p)},$$
where $s = s_p(m,i,j)$. The statement now follows from
Theorem~\ref{trthm}~(iii). 
\end{proof}

\begin{proof}[of Theorem~\ref{kthm}]The statement follows immediately
from the long exact sequence recalled in the introduction together with
Propositions~\ref{middleterm},~\ref{lefthandterm}, and~\ref{verschiebung}. 
\end{proof}

\section{The dual numbers}\label{dualnumberssection}

It follows from Theorem~\ref{trthm} that $K_{2i}(\Z[x]/(x^2),(x))$ is a
finite abelian group of order $(2i)!$. In this section, we investigate
the structure of these groups in low degrees.

\begin{theorem}\label{dualnumberthm}There are isomorphisms
$$\begin{aligned}
K_2(\Z[x]/(x^2),(x)) & \approx \Z/2\Z \cr
K_4(\Z[x]/(x^2),(x)) & \approx \Z/8\Z \oplus \Z/3\Z \cr
K_6(\Z[x]/(x^2),(x)) & \approx \Z/2\Z \oplus \Z/2\Z \oplus
\Z/4\Z \oplus \Z/9\Z \oplus \Z/5\Z \cr
\end{aligned}$$
\end{theorem}

\begin{proof}We know from Theorem~\ref{trthm} that the orders of these
three groups are as stated. Hence, it suffices to show that the
$2$-primary torsion subgroup of the groups in degree $4$ and $6$ and
the $3$-primary torsion subgroup of the group in degree $6$ are as
stated.

We first consider the group in degree $4$ which is given by the
short exact sequence
$$0 \to \lim_R \TR_{3-\lambda_d}^{r/2}(\Z) \xto{V_2}
\lim_R \TR_{3-\lambda_d}^r(\Z) \to
K_4(\Z[x]/(x^2),(x)) \to 0.$$
The middle term in the short exact sequence decomposes $2$-locally as
the direct sum
$$\lim_R \TR_{3-\lambda_d}^r(\Z)_{(2)} \xto{\sim}
\TR_{3-\lambda_1}^3(\Z;2)_{(2)} \oplus
\TR_{3-\lambda_1}^1(\Z;2)_{(2)},$$
where the first and second summands on the right-hand side correspond
to $j = 1$ and $j = 3$, respectively. Similarly, the left-hand term in
the short exact sequence decomposes $2$-locally as
$$\lim_R \TR_{3-\lambda_d}^{r/2}(\Z)_{(2)} \xto{\sim}
\TR_{3-\lambda_1}^2(\Z;2)_{(2)} \oplus
\TR_{3-\lambda_1}^0(\Z;2)_{(2)}.$$
The summands corresponding to $j = 3$ are both zero. Hence, the
$2$-primary torsion subgroup of $K_4(\Z[x]/(x^2),(x))$ is canonically
isomorphic to the cokernel of the Verschiebung map
$$V \colon \TR_{3-\lambda_1}^2(\Z;2)_{(2)} \to
\TR_{3-\lambda_1}^3(\Z;2)_{(2)}.$$
To evaluate this cokernel, we consider the following diagram
$$\xymatrix{
{ 0 } \ar[r] &
{ \TR_3^1(\Z;2)_{(2)} } \ar[r]^{V} \ar@{=}[d] &
{ \TR_3^2(\Z;2)_{(2)} } \ar[r]^(.45){\iota_*} \ar[d]^{V} &
{ \TR_{3-\lambda_1}^2(\Z;2)_{(2)} } \ar[r] \ar[d]^{V} &
{ 0 } \cr
{ 0 } \ar[r] &
{ \TR_3^1(\Z;2)_{(2)} } \ar[r]^{V^2} &
{ \TR_3^3(\Z;2)_{(2)} } \ar[r]^(.45){\iota_*} &
{ \TR_{3-\lambda_1}^3(\Z;2)_{(2)} } \ar[r] &
{ 0 } \cr
}$$
where the rows are exact by~\cite[Proposition~4.2]{h6}. It follows
from~\cite[Theorem~18]{h7} that this diagram is isomorphic to the
diagram
$$\xymatrix{
{ 0 } \ar[r] &
{ \Z/2\Z } \ar[r]^{4} \ar@{=}[d] &
{ \Z/8\Z } \ar[r]^{1} \ar[d]^{(a,b)} &
{ \Z/4\Z } \ar[r] \ar[d]^{(a,b)} &
{ 0 } \cr
{ 0 } \ar[r] &
{ \Z/2\Z } \ar[r]^(.37){(0,4)} &
{ \Z/8\Z \oplus \Z/8\Z } \ar[r]^{1 \oplus 1} &
{ \Z/8\Z \oplus \Z/4\Z } \ar[r] &
{ 0 }
}$$
where $a \in 2\Z/8\Z$ and $b \in (\Z/8\Z)^*$. Now, the cokernels of
the middle and right-hand vertical maps are isomorphic to
$\Z/8\Z$. This shows that $K_4(\Z[x]/(x^2),(x))_{(2)}$ is as stated. 

We next consider the group in degree $6$ which is given by the
short exact sequence
$$0 \to \lim_R \TR_{5-\lambda_d}^{r/2}(\Z) \xto{V_2}
\lim_R \TR_{5-\lambda_d}^r(\Z) \to
K_6(\Z[x]/(x^2),(x)) \to 0$$
and begin by evaluating the $2$-primary torsion subgroup. The middle
term in the short exact sequence decomposes $2$-locally as the direct
sum
$$\lim_R \TR_{5-\lambda_d}^r(\Z)_{(2)} \xto{\sim}
\TR_{5-\lambda_1}^3(\Z;2)_{(2)} \oplus
\TR_{5-\lambda_2}^2(\Z;2)_{(2)} \oplus
\TR_{5-\lambda_2}^1(\Z;2)_{(2)},$$
where the three summands on the right-hand side correspond
to $j = 1$, $j = 3$, and $j = 5$, respectively. Similarly, the
left-hand term in the short exact sequence decomposes $2$-locally as
$$\lim_R \TR_{5-\lambda_d}^{r/2}(\Z)_{(2)} \xto{\sim}
\TR_{5-\lambda_1}^2(\Z;2)_{(2)} \oplus
\TR_{5-\lambda_2}^1(\Z;2)_{(2)} \oplus
\TR_{5-\lambda_2}^0(\Z;2)_{(2)}.$$
The summands corresponding to $j = 5$ are both zero. Hence, the
$2$-primary torsion subgroup of $K_6(\Z[x]/(x^2),(x))$ is canonically
isomorphic to the direct sum of the cokernels of
$$\begin{aligned}
V & \colon \TR_{5-\lambda_1}^2(\Z;2)_{(2)} \to
\TR_{5-\lambda_1}^3(\Z;2)_{(2)} \cr
V & \colon \TR_{5-\lambda_2}^1(\Z;2)_{(2)} \to
\TR_{5-\lambda_2}^2(\Z;2)_{(2)}. \cr
\end{aligned}$$
We show that these are isomorphic to $\Z/2\Z \oplus \Z/2\Z$ and
$\Z/4\Z$, respectively. The statement for the latter cokernel follows
directly from Theorems~\ref{trthm} and~\ref{modp}. The two theorems
also show that the group $\TR_{5-\lambda_1}^2(\Z;2)_{(2)}$ is
isomorphic to $\Z/2\Z \oplus \Z/2\Z$ and that the group
$\TR_{5-\lambda_1}^3(\Z;2)_{(2)}$ is isomorphic to either $\Z/4\Z
\oplus \Z/4\Z$ or $\Z/8\Z \oplus \Z/2\Z$. We will prove that the
latter group is isomorphic to $\Z/4\Z \oplus \Z/4\Z$ by showing that
it contains $\Z/4\Z$ as a direct summand. To this end, we consider the
following diagram
$$\xymatrix{
{ 0 } \ar[r] &
{ \mathbb{H}_5(C_2,\TR_{\,\boldsymbol{\cdot}-\lambda_1}^1(\Z;2))_{(2)} }
\ar[r] \ar[d]^{V} &
{ \TR_{5-\lambda_1}^2(\Z;2)_{(2)} } \ar[r] \ar[d]^{V} &
{ \TR_5^1(\Z;2)_{(2)} } \ar[r] \ar[d]^{V} &
{ 0 } \cr
{ 0 } \ar[r] &
{ \mathbb{H}_5(C_4,\TR_{\,\boldsymbol{\cdot}-\lambda_1}^1(\Z;2))_{(2)} } \ar[r]
\ar[d]^{F^2} &
{ \TR_{5-\lambda_1}^3(\Z;2)_{(2)} } \ar[r] \ar[d]^{F^2} \ar[r] &
{ \TR_5^2(\Z;2)_{(2)} } \ar[r] &
{ 0 } \cr
{} &
{ \TR_3^1(\Z;2)_{(2)} } \ar@{=}[r] & 
{ \TR_3^1(\Z;2)_{(2)} } &
{} &
{} \cr
}$$
where the rows, but not the columns, are exact. It follows from
Theorem~\ref{trthm} that the top middle and right-hand vertical maps $V$
are injective. Hence, also the top left-hand vertical map $V$ is
injective. Moreover,~\cite[Proposition~4.2]{h6}
and~\cite[Proposition~15]{h7} show that the bottom left-hand vertical
map $F^2$ is surjective. Hence, also the bottom right-hand vertical map
$F^2$ is surjective. The skeleton spectral sequence
$$E_{s,t}^2 = H_s(C_4,\TR_{t-\lambda_1}^1(\Z;2))_{(2)} \Rightarrow
\mathbb{H}_{s+t}(C_4,\TR_{\,\boldsymbol{\cdot}-\lambda_1}^1(\Z;2))_{(2)}$$
shows that the middle left-hand group is an extension of
$E_{2,3}^{\infty} = \Z/4\Z$ by $E_{0,5}^{\infty} = \Z/2\Z$ and the
diagram above shows that the extension is split. It follows
from~\cite[Lemma~6]{h7} that
$$F \colon \mathbb{H}_5(C_4,\TR_{\,\boldsymbol{\cdot}-\lambda_1}^1(\Z;2))_{(2)} \to
\mathbb{H}_5(C_2,\TR_{\,\boldsymbol{\cdot}-\lambda_1}^1(\Z;2))_{(2)}$$
maps the generator of the summand $E_{0,5}^{\infty} = \Z/2\Z$ to
zero. Hence, the lower left-hand vertical map $F^2$ in the diagram
above maps the generator of the summand $E_{0,5}^{\infty} = \Z/2\Z$ to
zero. But the map $F^2$ is surjective, and therefore, maps a generator
of the summand $E_{3,2}^{\infty} = \Z/4\Z$ non-trivially. It follows
that $\TR_{5-\lambda_1}^3(\Z;2)_{(2)}$ contains direct summand
isomorphic to $\Z/4\Z$, and hence, is isomorphic to $\Z/4\Z \oplus 
\Z/4\Z$. This shows that the cokernel of the upper middle vertical map
$V$ is isomorphic to $\Z/2\Z \oplus \Z/2\Z$, and hence, that
$K_6(\Z[x]/(x^2),(x))_{(2)}$ is as stated. 

It remains to evaluate $K_6(\Z[x]/(x^2),(x))_{(3)}$. This group is
canonically isomorphic to the direct sum of
$\TR_{5-\lambda_1}^2(\Z;3)_{(3)}$ and
$\TR_{5-\lambda_2}^1(\Z;3)_{(3)}$. It follows from Theorem~\ref{trthm}
that the former group has order $9$ and that the latter group is
zero and from Theorem~\ref{modp} that the former group is cyclic. This
completes the proof.
\end{proof}

\begin{theorem}\label{oddtorsion}Let $p$ be an odd prime number. Then
for $2i < p^2$ the $p$-primary torsion subgroup of
$K_{2i}(\Z[x]/(x^2),(x))$ is isomorphic to $(\Z/p\Z)^{r_1} \oplus
(\Z/p^2\Z)^{r_2}$, where
$$(r_1,r_2) = \begin{cases}
(0,\lfloor i/p \rfloor) & \text{if $2i+1 \equiv 0$ modulo $p$} \cr
(\lfloor 2i/p \rfloor - 2,1) & \text{if $2i+1 \equiv j$ modulo $p$
with $1 \leqslant j \leqslant 2i/p$ odd} \cr
(\lfloor 2i/p \rfloor, 0) & \text{otherwise.} \cr
\end{cases}$$
Here $\lfloor x \rfloor$ denotes the largest integer less than or
equal to $x$.
\end{theorem}

\begin{proof}After localizing at the odd prime number $p$, the
short exact sequence
$$0 \to \lim_R \TR_{2i-1-\lambda_d}^{r/2}(\Z) \xto{V_2}
\lim_R \TR_{2i-1-\lambda_d}^r(\Z) \to K_{2i}(\Z[x]/(x^2),(x)) \to 0$$
induces a canonical isomorphism
$$\bigoplus_j \TR_{2i-1-\lambda_d}^s(\Z;p)_{(p)} \xto{\sim}
K_{2i}(\Z[x]/(x^2),(x))_{(p)},$$
where the sum runs over integers $1 \leqslant j \leqslant 2i$ coprime
to both $2$ and $p$ and $s = s_p(2,i,j)$ is the unique integer that
satisfies $p^{s-1}j \leqslant 2i < p^sj$. Since $2i < p^2$ we find
$$s_p(2,i,j) = \begin{cases}
2 & \text{if $1 \leqslant j \leqslant 2i/p$} \cr
1 & \text{if $2i/p < j \leqslant 2i$.} \cr
\end{cases}$$
If $2i/p < j \leqslant 2i$ we have $d = (j-1)/2$, and hence
$$\length_{\Z_{(p)}} \TR_{2i-1-\lambda_d}^1(\Z;p)_{(p)} = v_p(i-d) =
v_p(2i+1-j).$$
The length is at most $1$ since $2i < p^2$. If
$1 \leqslant j \leqslant 2i/p$ we have $d = (pj-1)/2$, and in this
case Theorem~\ref{trthm} shows that
$$\begin{aligned}
\length_{\Z_{(p)}} \TR_{2i-1-\lambda_d}^2(\Z;p)_{(p)} & = 
\length_{\Z_{(p)}} \TR_{2i-1-\lambda_d'}^1(\Z;p)_{(p)} + v_p(i-d) + 1 \cr
{} & = v_p(2i+1-j) + v_p(2i+1) + 1, \cr
\end{aligned}$$
where we have used that $\lambda_d' = \lambda_{d'}$ with $d' =
(j-1)/2$. The length is at most $2$ since $2i < p^2$ and since $j$
is coprime to $p$. We claim that the group
$\TR_{2i-1-\lambda_d}^2(\Z;p)_{(p)}$ is always cyclic. By
Theorem~\ref{modp}, the claim is equivalent to the congruence
$$2i-1 \not\equiv 2\delta_p(\lambda_d)-1 \mod 2p^2.$$
We compute that modulo $p^2$
$$\delta_p(\lambda_d) \equiv (1-p)((pj-1)/2 + (j-1)/2 \cdot p)$$
Hence, we have $2i-1 \equiv 2\delta_p(\lambda_d)-1$ modulo $2p^2$ if
and only if $2i + 1 \equiv 2pj + p^2$ modulo $2p^2$. This is possible
only if $2i+1$ is congruent to $0$ modulo $p$. If we write $2i+1 =
ap$ then $1 \leqslant a \leqslant p$ and $1 \leqslant j < a$. Hence,
$p \leqslant ap \leqslant p^2$ and $2p + p^2 \leqslant 2pj + p^2 < 2ap
+ p^2 \leqslant 3p^2$ which implies that the congruence $ap \equiv 2pj
+ p^2$ modulo $2p^2$ is equivalent to the equality $ap+p^2 = 2pj$. But
then $a+p = 2j$ which contradicts that $j < a$. The claim follows.

We now show that the integers $(r_1,r_2)$ are as stated. It follows
from Theorem~\ref{kthm} that
$$r_1 + 2r_2 = v_p((2i)!) = \lfloor 2i/p \rfloor.$$
In the case where $2i+1$ is congruent to $0$ modulo $p$, we proved
above that $r_1 = 0$. If we write $2i+1 = ap$ then $a$ is odd and
$$r_2 = \lfloor 2i/p \rfloor / 2 = (a-1)/2 = \lfloor (a-1)/2 + (p-1)/2p
\rfloor = \lfloor i/p \rfloor$$
as stated. In the case where $2i+1 \equiv j$ modulo $p$ with $1
\leqslant j \leqslant 2i/p$ odd, we proved above that $r_2 =
1$. Hence, $r_1 = \lfloor 2i/p \rfloor - 2$. Finally, in the remaining
case, $r_2 = 0$, and hence, $r_1 = \lfloor 2i/p \rfloor$. This
completes the proof.
\end{proof}

\begin{example}\label{vanishing}Let $p$ be an odd prime number. We
spell out the statement of Theorem~\ref{oddtorsion}, for
$i \leqslant p+1$. The $p$-primary torsion subgroup of
$K_{2i}(\Z[x]/(x^2),(x))$ is zero for $i \leqslant (p-1)/2$, is
cyclic of order $p$ for $(p+1)/2 \leqslant i < p$, and is cyclic of
order $p^2$ for $i = p$. The structure of the $p$-primary torsion
subgroup of $K_{2p+2}(\Z[x]/(x^2),(x))$ depends on the odd prime
$p$. It is cyclic of order $9$ if $p = 3$, and the direct sum of two
cyclic groups of order $p$ if $p \geqslant 5$.
\end{example}

\providecommand{\bysame}{\leavevmode\hbox to3em{\hrulefill}\thinspace}
\providecommand{\MR}{\relax\ifhmode\unskip\space\fi MR }
\providecommand{\MRhref}[2]{%
  \href{http://www.ams.org/mathscinet-getitem?mr=#1}{#2}
}
\providecommand{\href}[2]{#2}

\begin{acknowledgements}This paper was written in part while the third
author visited the Hausdorff Center for Mathematics at the University
of Bonn. He would like to thank the center and in particular Stefan
Schwede and Christian Ausoni for their hospitality and
support. Finally, we thank an anonymous referee for helpful comments.
\end{acknowledgements}

\affiliationone{
Vigleik Angeltveit\\
Department of Mathematics\\
The University of Chicago\\
Chicago, Illinois\\
USA\\
\email{vigleik@math.uchicago.edu}
}
\affiliationtwo{
Teena Gerhardt\\
Department of Mathematics\\
Indiana University\\
Bloomington, Indiana\\
USA\\
\email{tgerhard@indiana.edu}
}
\affiliationthree{
Lars Hesselholt\\
Graduate School of Mathematics\\
Nagoya University\\
Chikusa-ku\\
Nagoya, Japan
Japan\\
\email{larsh@math.nagoya-u.ac.jp}
}

\end{document}